\documentclass[preprint,12pt]{elsarticle}

\usepackage{mathrsfs}
\usepackage{amsmath}
\usepackage{amssymb}
\usepackage{amsthm}
\usepackage{txfonts}
\usepackage{amsfonts}
\usepackage{latexsym}
\usepackage{indentfirst}

\usepackage{wrapfig}

\newcommand\RR{\mathbb{R}}
\newcommand\Rd{\mathbb{R}^d}
\newcommand\NN{\mathbb{N}}

\newcommand\vx{\boldsymbol{x}}   % x in R^d
\newcommand\vy{\boldsymbol{y}}   % y in R^d
\newcommand\vz{\boldsymbol{z}}
\newcommand\vc{\boldsymbol{c}}

\newcommand\vphi{\boldsymbol{\phi}}

\newcommand\valpha{\boldsymbol{\alpha}}
\newcommand\vbeta{\boldsymbol{\beta}}
\newcommand\vu{\boldsymbol{u}}

\newcommand\vb{\boldsymbol{b}}

\newcommand\vd{\boldsymbol{d}}

\newcommand\vv{\boldsymbol{v}}

\newcommand\vvartheta{\boldsymbol{\vartheta}}

\newcommand\vA{\mathsf{A}}
\newcommand\vB{\mathsf{B}}

\newcommand\Sset{\mathcal{S}}

\newcommand\Banach{\mathcal{B}}
\newcommand\Cont{\mathrm{C}}
\newcommand\Leb{\mathrm{L}}

\newcommand{\norm}[1]{\left\lVert#1\right\rVert}
\newcommand{\abs}[1]{\left\lvert#1\right\rvert}

\newtheorem{theorem}{Theorem}[section]

\newtheorem{proposition}[theorem]{Proposition}

\theoremstyle{definition}
\newtheorem{definition}[theorem]{Definition}

\theoremstyle{remark}
\newtheorem{remark}[theorem]{Remark}

\numberwithin{equation}{section}
\numberwithin{figure}{section}
\numberwithin{table}{section}
\begin{document}

\begin{frontmatter}

\title{Positive Definite Multi-Kernels for Scattered Data Interpolations}

\author[add1,add2]{Qi Ye}\ead{yeqi@m.scnu.edu.cn}

\address[add1]{School of Mathematical Sciences, South China Normal University, \\ Guangzhou, Guangdong, China, 510631}
\address[add2]{Pazhou Lab, Guangzhou, Guangdong, China, 510300}

\date{}

\begin{abstract}
In this article, we use the knowledge of positive definite tensors to develop a concept of positive definite multi-kernels to construct the kernel-based interpolants of scattered data.
By the techniques of reproducing kernel Banach spaces, the optimal recoveries and error analysis of the kernel-based interpolants are shown for a special class of strictly positive definite multi-kernels.
\end{abstract}


\end{frontmatter}



%---------------------------------------------------------------------------------------------------------------------
\section{Introduction}\label{sec:Intr}
%---------------------------------------------------------------------------------------------------------------------

In many areas of practical applications, we often face a problem of reconstructing an unknown function $f$ from the scattered data.
The scattered data consist of high-dimension data points $\left\{\vx_i\right\}_{i=1}^n$ and data values $\left\{y_i\right\}_{i=1}^n$
such that $y_i=f(\vx_i)$ for all $i=1,2,\ldots,n$.
The reconstruction is to find an estimate function $s$ to approximate $f$.
Generally, $s$ is sought to interpolate the scattered data, that is, $s(\vx_i)=y_i$ for all $i=1,2,\ldots,n$.
It is well-known that the kernel-based approximation method is a fundamental approach of scattered data interpolations.
The classical kernel-based approximation method is mainly dependent of the positive definite kernels.
Here, we develop a concept of positive definite multi-kernels in Definition \ref{def:PDTK} which is a generalization of positive definite kernels.
The positive definite multi-kernels can be also used to reconstruct $f$ from the scattered data.
For examples, we compare the interpolations by a positive definite kernel $K_2:\otimes_{k=1}^2\Rd\to\RR$ and a positive definite multi-kernel
$K_4:\otimes_{k=1}^4\Rd\to\RR$.
%Let a matrix $\vA_2:=\left(K_2(\vx_{i_1},\vx_{i_2})\right)_{i_1,i_2=1}^{n,n}\in\RR^{n\times n}$
%and a tensor $\vA_4:=\left(K_4(\vx_{i_1},\vx_{i_2},\vx_{i_3},\vx_{i_4})\right)_{i_1,i_2,i_3,i_4=1}^{n,n,n,n}\in\RR^{n\times n\times n\times n}$.
%Thus, $\vA_2$ is a positive definite matrix and $\vA_4$ is a positive definite tensor.
%Let $\vc:=\left(c_1,\ldots,c_n\right)^T\in\RR^n$ and $\vy:=\left(y_1,\ldots,y_n\right)^T\in\RR^n$.
The classical interpolant $s_2$ is composed of a kernel basis
\[
K_2(\cdot,\vx_i),\quad\text{for }i=1,2,\ldots,n,
\]
that is,
\[
s_2(\vx):=\sum_{i=1}^nc_iK_2(\vx,\vx_i),\quad\text{for }\vx\in\Rd,
\]
where the coefficients $c_1,c_2,\ldots,c_n$ are solved by a linear system
\[
\sum_{i_2=1}^nK_2(\vx_{i_1},\vx_{i_2})c_{i_2}=y_{i_1},\quad\text{for }i_1=1,2,\ldots,n.
\]
%\[
%\vA_2\vc=\vy.
%\]
The new interpolant $s_4$ is composed of another kernel basis
\[
K_4(\cdot,\vx_{i_1},\vx_{i_2},\vx_{i_3}),\quad\text{for }i_1,i_2,i_3=1,2,\ldots,n,
\]
that is,
\[
s_4(\vx):=\sum_{i_1,i_2,i_3=1}^{n,n,n}c_{i_1}c_{i_2}c_{i_3}K_4(\vx,\vx_{i_1},\vx_{i_2},\vx_{i_3}),\quad\text{for }\vx\in\Rd,
\]
where the coefficients $c_1,c_2,\ldots,c_n$ are solved by a multi-linear system
\[
\sum_{i_2,i_3,i_4=1}^{n,n,n}K_4(\vx_{i_1},\vx_{i_2},\vx_{i_3},\vx_{i_4})c_{i_2}c_{i_3}c_{i_4}=y_{i_1},
\quad\text{for }i_1=1,2,\ldots,n.
\]
%\[
%\vA_4\vc^3=\vy.
%\]
In this article, we mainly discuss how to use a special class of strictly positive definite multi-kernel $\Phi_m$ in Equation \eqref{eq:PDTK-sp} to construct the kernel-based interpolant $s_m$ in Theorem \ref{thm:PDTK-Int} which shows that the related multi-linear system exists the unique solution.
By the theorems of reproducing kernel Banach spaces in \cite{XuYe2019}, we can obtain the advanced properties of $s_m$
including optimal recoveries in Theorem \ref{thm:PDTK-Int-Opt} and error analysis
in Theorems \ref{thm:PDTK-Int-Error} and \ref{thm:PDTK-Int-Error-h}.

%---------------------------------------------------------------------------------------------------------------------
\section{Positive Definite Tensors and Multi-Linear Systems}\label{sec:PDT-MLS}
%---------------------------------------------------------------------------------------------------------------------

In this section, we review the theory of positive definite tensors which will be used to define the positive definite multi-kernels in Section \ref{sec:PDTK-RKBS}.
For convenience of the readers, the notations and operations of tensors are defined as in the book \cite{QiChenChen2019}.
We say $T_{m,n}$ a collection of all $m$th order $n$th dimensional real tensors, for example, $T_{3,n}=\RR^{n\times n\times n}$, $T_{2,n}=\RR^{n\times n}$, and $T_{1,n}=\RR^{n}$.
For $\vA_m\in T_{m,n}$, we say $\vA_m$ a symmetric tensor if all entries $a_{i_1\ldots i_m}$ of $\vA_m$ are invariant under any permutation of the indices.
%For $\vA_m\in T_{m,n}$, we say $\vA_m$ a nonnegative tensor described as $\vA_m\geq0$ if all entries $a_{i_1\ldots i_m}$ of $\vA_m$ are nonnegative.
%For $\vA_m,\vB_m\in T_{m,n}$, we say $\vA_m\geq\vB_m$ if $\vA_m-\vB_m\geq0$.
%For $\vv\in\RR^n$, the tensor outer product is defined as
%\[
%\vv^{\otimes m}=\vv\otimes\cdots\otimes\vv:=\left(v_{i_1}\ldots v_{i_m}\right)\in T_{m,n}.
%\]
For $\vc\in\RR^n$, the tensor multiplications are defined as
\[
\vA_{m}\vc^{m-1}:=\left(\sum_{i_2,\ldots,i_m=1}^{n,\ldots,n}a_{i_1i_2\ldots i_m}c_{i_2}\ldots c_{i_{m}}\right)_{i_1=1}^{n}\in\RR^n,
\]
and
\[
\vA_{m}\vc^{m}:=\sum_{i_1,\ldots,i_m=1}^{n,\ldots,n}a_{i_1\ldots i_m}c_{i_1}\ldots c_{i_{m}}\in\RR.
\]
For $\vA_m\in T_{m,n}$, we say $\vA_m$ a \emph{semi-positive definite tensor} if
\[
\vA_m\vc^m\geq0,\quad\text{for all }\vc\in\RR^n.
\]
If further $\vA_m\vc^m=0$ if and only if $\vc=0$,
then $\vA_m$ is said a \emph{positive definite tensor} same as in \cite[Definition 4.1]{WangHuangQi2018}.
Specially, the positive definite tensor $\vA_2$ is a positive definite matrix.

For any $\vb\in\RR^n$, we look at
the multi-linear system
\begin{equation}\label{eq:Amsys}
\vA_m\vc^{m-1}=\vb.
\end{equation}
If $\vA_2$ is a symmetric positive definite matrix, then Equation \eqref{eq:Amsys}
has a unique solution $\vc\in\RR^n$.
For the general cases, we will study the solutions of Equation \eqref{eq:Amsys} by the theorems of the tensor variational inequality (TVI) in \cite{WangHuangQi2018}. Let
\[
\Sset(\vA_m,\vb):=\left\{\vc\in\RR^n: \left(\vc-\vd\right)^T\left(\vA_m\vc^{m-1}-\vb\right)\geq0,
\text{ for all }\vd\in\RR^n
\right\}.
\]
Then $\Sset(\vA_m,\vb)$ is said the solutions set of TVI with $\RR^n$, $\vA_m$, and $-\vb$.
If $\vA_m$ is a positive definite tensor, then \cite[Theorem 4.2]{WangHuangQi2018} guarantees that $\Sset(\vA_m,\vb)$
is nonempty and compact.

%%%%%%%%%%%%%%%%%%%%%%%%%%%%%%%%%%%%%%%%%%%%%%%%%%%%%%%%%%%%%%%%%%%%%%%%%%%%%%%%%%%%%
\begin{proposition}\label{pro:PDT-MultiSys}
If $\vA_m$ is a positive definite tensor, then the solutions set of Equation \eqref{eq:Amsys} is nonempty and compact.
\end{proposition}
%%%%%%%%%%%%%%%%%%%%%%%%%%%%%%%%%%%%%%%%%%%%%%%%%%%%%%%%%%%%%%%%%%%%%%%%%%%%%%%%%%%%%
\begin{proof}
If we verify that $\Sset(\vA_m,\vb)$ is equal to the solutions set of Equation \eqref{eq:Amsys}, then the proof is completed by \cite[Theorem 4.2]{WangHuangQi2018}.

We first suppose that $\vc$ is a solution of Equation \eqref{eq:Amsys}. Thus, we conclude from $\vA_m\vc^{m-1}-\vb=0$ that $\left(\vc-\vd\right)^T\left(\vA_m\vc^{m-1}-\vb\right)=0$ for all $\vd\in\RR^n$, hence that $\vc\in \Sset(\vA_m,\vb)$.
%
%In the other side,
Moreover, we suppose that $\vc\in\Sset(\vA_m,\vb)$. Let $\vv:=\vc-\vd$ and $\vu:=\vA_m\vc^{m-1}-\vb$. Thus,
$\vv^T\vu\geq0$ for any $\vv\in\RR^n$.
Since the range of $\vv$ is equal to $\RR^n$,
we take $\vv:=-\vu$ such that $0\geq-\norm{\vu}_2^2=(-\vu)^T\vu\geq0$. This shows that $\vA_m\vc^{m-1}-\vb=\vu=0$. Therefore, $\vc$ is a solution of Equation \eqref{eq:Amsys}.
\end{proof}

Same as in \cite[Definition 4.1]{WangHuangQi2018}, we say $\vA_m$ a \emph{strictly positive definite tensor}
if $\left(\vc-\vd\right)^T\left(\vA_m\vc^{m-1}-\vA_m\vd^{m-1}\right)>0$ for all $\vc,\vd\in\RR^n$ with $\vc\neq\vd$.
Obviously, if $\vA_m$ is a strictly positive definite tensor, then $\vA_m$ is a positive definite tensor.
Specially, the positive definite matrix $\vA_2$ is a strictly positive definite tensor.
If $\vA_m$ is a strictly positive definite tensor, then \cite[Theorem 4.2]{WangHuangQi2018} guarantees that $\Sset(\vA_m,\vb)$
is singleton.

%%%%%%%%%%%%%%%%%%%%%%%%%%%%%%%%%%%%%%%%%%%%%%%%%%%%%%%%%%%%%%%%%%%%%%%%%%%%%%%%%%%%%
\begin{proposition}\label{pro:SPDT-MultiSys}
If $\vA_m$ is a strictly positive definite tensor, then Equation \eqref{eq:Amsys} exists a unique solution.
\end{proposition}
%%%%%%%%%%%%%%%%%%%%%%%%%%%%%%%%%%%%%%%%%%%%%%%%%%%%%%%%%%%%%%%%%%%%%%%%%%%%%%%%%%%%%
\begin{proof}
As in the proof of Proposition \ref{pro:PDT-MultiSys}, $\Sset(\vA_m,\vb)$ is equal to the solutions set of Equation \eqref{eq:Amsys}. Thus, the proof is completed by \cite[Theorem 4.2]{WangHuangQi2018}.
\end{proof}

In \cite{DingWei2016,YanLingLingHe2019},
there are many kinds of multi-linear systems such as $M$-tensors, $Z^+$-tensors, and $P$-tensors to introduce the same results.
In this article, we only discuss the positive definite tensors for the scattered data interpolations.

%---------------------------------------------------------------------------------------------------------------------
\section{Positive Definite Multi-Kernels and Reproducing Kernel Banach Spaces}\label{sec:PDTK-RKBS}
%---------------------------------------------------------------------------------------------------------------------

In this section, we develop a concept of positive definite multi-kernels.
Let a domain $\Omega\subseteq\Rd$.
By \cite[Definition~6.24]{Wendland2005}, a symmetric kernel $K_2:\otimes_{k=1}^2\Omega\to\RR$ is said a positive definite kernel
if, for all $n\in\NN$ and all pairwise distinct points $\left\{\vx_i\right\}_{i=1}^n\subseteq\Omega$, the quadratic form
\[
\sum_{i_1,i_2=1}^{n,n}K_2(\vx_{i_1},\vx_{i_2})c_{i_1}c_{i_2}>0,
\]
for all $\vc:=(c_1,c_2,\cdots,c_n)^T\in\RR^n\backslash\{0\}$.
Let $\vA_2:=\left(K_2(\vx_{i_1},\vx_{i_2})\right)_{i_1,i_2=1}^{n,n}\in T_{2,n}$.
Thus, $\vA_2$ is a symmetric positive definite matrix.

A multi-kernel of order $m\in\NN$ is defined as $K_m:\otimes_{k=1}^m\Omega\to\RR$ for $m\in\NN$.
We say $K_m$ a \emph{symmetric multi-kernel} if $K_m(\vz_1,\vz_2,\cdots,\vz_m)=K_m(\vz_{i_1},\vz_{i_2},\cdots,\vz_{i_m})$
for all permutation of the pairwise distinct indices $i_1,i_2,\ldots,,i_m\in\{1,2,\ldots,m\}$.

\begin{definition}\label{def:PDTK}
A multi-kernel $K_m$ is said a \emph{semi-positive definite multi-kernel} if, for all $n\in\NN$ and all pairwise distinct points $\left\{\vx_i\right\}_{i=1}^n\subseteq\Omega$, the multiple product form
\[
\sum_{i_1,i_2,\ldots,i_m=1}^{n,n,\ldots,n}K_m(\vx_{i_1},\vx_{i_2},\cdots,\vx_{i_m})c_{i_1}c_{i_2}\cdots c_{i_{m}}
\geq0,
\]
for all $\vc:=(c_1,c_2,\cdots,c_n)^T\in\RR^n$.
If further the multiple product is equal to 0
%\[
%\sum_{i_1,i_2,\ldots,i_m=1}^{n,n,\ldots,n}K_m(\vx_{i_1},\vx_{i_2},\cdots,\vx_{i_m})c_{i_1}c_{i_2}\cdots c_{i_{m}}=0,
%\]
if and only if $\vc=0$,
then $K_m$ is said a \emph{positive definite multi-kernel}.

A symmetric multi-kernel $K_m$ is said a \emph{strictly positive definite multi-kernel} if, for all $n\in\NN$ and all pairwise distinct points $\left\{\vx_i\right\}_{i=1}^n\subseteq\Omega$, the multiple product form
\[
\sum_{i_1,i_2,\ldots,i_m=1}^{n,n,\ldots,n}K_m(\vx_{i_1},\vx_{i_2},\cdots,\vx_{i_m})\left(c_{i_1}-d_{i_1}\right)\left(c_{i_2}\cdots c_{i_{m}}-d_{i_2}\cdots d_{i_{m}}\right)
>0,
\]
for all $\vc:=(c_1,c_2,\cdots,c_n)^T,\vd:=(d_1,d_2,\cdots,d_n)^T\in\RR^n$ with $\vc\neq\vd$.
\end{definition}

Obviously, the symmetric strictly positive definite multi-kernel $K_m$ is a symmetric positive definite multi-kernel.
Specially, the symmetric positive definite multi-kernel $K_2$ is a symmetric positive definite kernel and a symmetric strictly positive definite multi-kernel.
Given any pairwise distinct points $\left\{\vx_i\right\}_{i=1}^n\subseteq\Omega$, we define
\begin{equation}\label{eq:Am-K}
\vA_m:=\left(K_m(\vx_{i_1},\vx_{i_2},\cdots,\vx_{i_m})\right)_{i_1,i_2,\ldots,i_m=1}^{n,n,\ldots,n}\in T_{m,n}.
\end{equation}

%%%%%%%%%%%%%%%%%%%%%%%%%%%%%%%%%%%%%%%%%%%%%%%%%%%%%%%%%%%%%%%%%%%%%%%%%%%%%%%%%%%%%
\begin{proposition}\label{pro:PDTK-PDT}
If $K_m$ is a symmetric (strictly) positive definite multi-kernel, then $\vA_m$ in Equation \eqref{eq:Am-K} is a symmetric (strictly) positive definite tensor.
\end{proposition}
%%%%%%%%%%%%%%%%%%%%%%%%%%%%%%%%%%%%%%%%%%%%%%%%%%%%%%%%%%%%%%%%%%%%%%%%%%%%%%%%%%%%%
\begin{proof}
The proof is straightforward by the definitions of (strictly) positive definite multi-kernels and (strictly) positive definite tensors.
\end{proof}

Now we look at a special class of multi-kernels related to the $p$-norm reproducing kernel Banach spaces in \cite{Wendland2005}. Same as in \cite[Section 10.4]{Wendland2005}, we suppose that the domain $\Omega$ is \emph{compact}
and the symmetric positive definite kernel $\Phi_2$ is \emph{continuous}.
Thus, the Mercer theorem assures that $\Phi_2$ possesses the absolutely uniformly convergent representation
\[
\Phi_2(\vz_1,\vz_2)=\sum_{k=1}^{\infty}\lambda_ke_k(\vz_1)e_k(\vz_2),\quad\text{for }\vz_1,\vz_2\in\Omega,
\]
where $\left\{\lambda_k:k\in\NN\right\}$ and $\left\{e_k:k\in\NN\right\}$
are the related positive eigenvalues and continuous eigenfunctions of $\Phi_2$, respectively.
Let $\phi_k:=\lambda_k^{1/2}e_k$ for all $k\in\NN$.
Same as \cite[Section 4.2]{XuYe2019}, we further suppose that
\begin{equation}\label{eq:cond-phi}
\sum_{k=1}^{\infty}\abs{\phi_k(\vx)}<\infty,\quad\text{for all }\vx\in\Omega.
\end{equation}
Thus, by \cite[Theorem 5.10]{XuYe2019}, Equation \eqref{eq:cond-phi} assures that the \emph{special} multi-kernel
\begin{equation}\label{eq:PDTK-sp}
\Phi_m(\vz_1,\vz_2,\cdots,\vz_m):=\sum_{k=1}^{\infty}\prod_{i=1}^m\phi_k(\vz_i),\quad\text{for }\vz_1,\vz_2,\ldots,\vz_m\in\Omega,
\end{equation}
is well-defined for all $m\geq2$.

\begin{remark}
In this article, the positive definite multi-kernels are consistent with the positive definite tensors in \cite{QiChenChen2019,WangHuangQi2018}.
The semi-positive definite kernels discussed here have the same meaning as the positive definite kernels in \cite{XuYe2019}.
For convenience, we will \emph{always} suppose that $K_2$ is a continuous symmetric positive definite kernel on the compact domain $\Omega$.
Moreover, the eigenvalues and eigenfunctions of $\Phi_2$ are fixed and satisfy the condition in Equation \eqref{eq:cond-phi} such that $\Phi_m$ is fixed and well-defined in the following discussions.
\end{remark}

%%%%%%%%%%%%%%%%%%%%%%%%%%%%%%%%%%%%%%%%%%%%%%%%%%%%%%%%%%%%%%%%%%%%%%%%%%%%%%%%%%%%%
\begin{proposition}\label{pro:Phi-sPDT}
If $m$ is a positive even integer and $K_m:=\Phi_m$, then $\vA_m$ in Equation \eqref{eq:Am-K} is a symmetric strictly positive definite tensor.
\end{proposition}
%%%%%%%%%%%%%%%%%%%%%%%%%%%%%%%%%%%%%%%%%%%%%%%%%%%%%%%%%%%%%%%%%%%%%%%%%%%%%%%%%%%%%
\begin{proof}
Equation \eqref{eq:PDTK-sp} shows that $\Phi_m$ is a symmetric multi-kernel.
Thus, $\vA_m$ is a symmetric tensor.
We take any $\vc,\vd\in\RR^n$ with $\vc\neq\vd$. By the definition of strictly positive definite tensors, if we prove that
$(\vc-\vd)^T\left(\vA_m\vc^{m-1}-\vA_m\vd^{m-1}\right)>0$,
then the proof is complete.

Let $\vv_k:=(\phi_k(\vx_1),\phi_k(\vx_2),\cdots,\phi_k(\vx_n))^T$ for all $k\in\NN$. We conclude from Equation \eqref{eq:PDTK-sp} that
$\vA_m=\sum_{k=1}^{\infty}\vv_k^{\otimes m}$, hence that
$\vA_m\vc^{m-1}=\sum_{k=1}^{\infty}(\vv_k^T\vc)^{m-1}\vv_k$ and
$\vA_m\vd^{m-1}=\sum_{k=1}^{\infty}(\vv_k^T\vd)^{m-1}\vv_k$.
Let $\alpha_k:=\vv_k^T\vc$ and $\beta_k:=\vv_k^T\vd$ for all $k\in\NN$.
Thus,
\[
(\vc-\vd)^T\left(\vA_m\vc^{m-1}-\vA_m\vd^{m-1}\right)
=\sum_{k=1}^{\infty}\left(\alpha_k-\beta_k\right)\left(\alpha_k^{m-1}-\beta_k^{m-1}\right).
\]
Now we verify that $\left(\alpha_k-\beta_k\right)\left(\alpha_k^{m-1}-\beta_k^{m-1}\right)$ is nonnegative for all $k\in\NN$.
Since $m$ is even, $m-1$ is odd.
If $\alpha_k\geq0$ and $\beta_k\geq0$,
then $\alpha_k-\beta_k$ and $\alpha_k^{m-1}-\beta_k^{m-1}$ have the same sign.
If $\alpha_k\geq0$ and $\beta_k\leq0$, then $\alpha_k-\beta_k\geq0$ and $\alpha_k^{m-1}-\beta_k^{m-1}\geq0$.
Thus, $\left(\alpha_k-\beta_k\right)\left(\alpha_k^{m-1}-\beta_k^{m-1}\right)\geq0$.
Moreover, since $\left(\alpha_k-\beta_k\right)\left(\alpha_k^{m-1}-\beta_k^{m-1}\right)=\left(\beta_k-\alpha_k\right)\left(\beta_k^{m-1}-\alpha_k^{m-1}\right)$, the assertion also holds. Therefore, we conclude that
$(\vc-\vd)^T\left(\vA_m\vc^{m-1}-\vA_m\vd^{m-1}\right)\geq0$.

Since $\Phi_2$ is a symmetric positive definite kernel,
$\vA_2$ is a symmetric positive definite matrix.
Thus, we conclude from $\sum_{k=1}^{\infty}\left(\alpha_k-\beta_k\right)^2=(\vc-\vd)^T\left(\vA_2\vc-\vA_2\vd\right)=\vA_2(\vc-\vd)^2>0$
that $\alpha_k-\beta_k\neq0$ such that $\left(\alpha_k-\beta_k\right)\left(\alpha_k^{m-1}-\beta_k^{m-1}\right)\neq0$ for at least one $k$,
hence that $(\vc-\vd)^T\left(\vA_m\vc^{m-1}-\vA_m\vd^{m-1}\right)$ is positive.
\end{proof}

%%%%%%%%%%%%%%%%%%%%%%%%%%%%%%%%%%%%%%%%%%%%%%%%%%%%%%%%%%%%%%%%%%%%%%%%%%%%%%%%%%%%%
\begin{proposition}\label{pro:Phi-PDTK}
If $m$ is a positive even integer, then $\Phi_m$ is a symmetric strictly positive definite multi-kernel.
\end{proposition}
%%%%%%%%%%%%%%%%%%%%%%%%%%%%%%%%%%%%%%%%%%%%%%%%%%%%%%%%%%%%%%%%%%%%%%%%%%%%%%%%%%%%%
\begin{proof}
Proposition \ref{pro:Phi-sPDT} guarantees that $\vA_m$ is a symmetric strictly positive definite tensor.
Thus, the proof is straightforward by Definition \ref{def:PDTK}.
\end{proof}

%
%%%%%%%%%%%%%%%%%%%%%%%%%%%%%%%%%%%%%%%%%%%%%%%%%%%%%%%%%%%%%%%%%%%%%%%%%%%%%%%%%%%%%%
%\begin{example}
%Let
%\[
%\Phi_2(\vz_1,\vz_2):=\prod_{j=1}^{d}R(z_{1j},z_{2j}),
%\]
%for $\vz_1:=(z_{11},z_{12},\cdots,z_{1d})^T,\vz_2:=(z_{21},z_{22},\cdots,z_{2d})^T\in(0,1)^d$,
%where
%\[
%R(z_1,z_2):=
%\begin{cases}
%-z_1^3+z_1^3z_2+z_1z_2^3-3z_1z_2^2+2z_1z_2,&0<z_1\leq z_2<1,\\
%-z_2^3+z_1z_2^3+z_1^3z_2-3z_1^2z_2+2z_1z_2,&0<z_2\leq z_1<1.
%\end{cases}
%\]
%Thus, $\Phi_2$ is a positive definite kernel on $(0,1)^d$. Let
%\[
%\lambda_{\vk}:=\prod_{j=1}^{d}\frac{6}{k_j^4\pi^4},\quad
%e_{\vk}(\vx):=2^{d/2}\prod_{j=1}^{d}\sin(k_j\pi x_j),
%\]
%for all $\vk:=(k_1,k_2,\cdots,k_d)^T\in\NN^d$.
%Thus, $\lambda_{\vk}$ and $e_{\vk}$ are the eigenvalues and eigenfunctions of $K_2$ and
%$\phi_{\vk}:=\lambda_{\vk}^{1/2}e_{\vk}$ satisfy the condition in Equation \eqref{eq:cond-phi}.
%By Proposition \ref{pro:Phi-PDTK}, we obtain the strictly positive definite multi-kernel
%\[
%\Phi_m(\vz_1,\vz_2,\cdots,\vz_m):=
%\sum_{\vk\in\NN^d}\prod_{i=1}^{m}\phi_{\vk}(\vz_i),
%\quad\text{for }\vz_1,\vz_2,\cdots,\vz_m\in(0,1)^d.
%\]
%\end{example}
%%%%%%%%%%%%%%%%%%%%%%%%%%%%%%%%%%%%%%%%%%%%%%%%%%%%%%%%%%%%%%%%%%%%%%%%%%%%%%%%%%%%%%
%
%Another examples of strictly positive definite multi-kernels can be constructed by Gaussian kernels and power series kernels such as in \cite[Sections 4.4 and 4.5]{XuYe2019}.
%
For examples, the strictly positive definite multi-kernels can be constructed by Gaussian kernels and power series kernels such as in \cite[Sections 4.4 and 4.5]{XuYe2019}.

Finally, we introduce the reproducing kernel Banach spaces defined by the positive definite kernels such as in \cite[Chapter 4]{XuYe2019}.
Since $\Phi_2$ is a symmetric positive definite kernel, the eigenfunctions $\left\{e_k:k\in\NN\right\}$ of $\Phi_2$ can be an orthonormal basis of $\Leb_2(\Omega)$.
Thus, $\left\{\phi_k:k\in\NN\right\}$ is a basis of $\Leb_2(\Omega)$.
%Since $K_2$ is equal to $\Phi_2$ here, we will use $\Phi_2$ to replace $K_2$ for simplification.
Let $\vphi:=\left(\phi_1,\phi_2,\cdots,\phi_k,\cdots\right)^T$. We define a normed space
\[
\Banach^{p}:=\left\{f:=\valpha^T\vphi:\valpha\in\RR^{\infty}\text{ and }\norm{\valpha}_{p}<\infty\right\},
\]
equipped with the norm $\norm{f}_{\Banach^{p}}:=\norm{\valpha}_{p}$, where $1<p<\infty$.
By Equation \eqref{eq:cond-phi}, \cite[Theorem 4.3]{XuYe2019} guarantees that
$\Banach^{p}$ is a two-sided reproducing kernel Banach space and $\Phi_2$ is the two-sided reproducing kernel of $\Banach^{p}$,
more precisely,
(i) $\Phi_2(\vz_1,\cdot)\in\Banach^{q}$, (ii) $\langle f,\Phi_2(\vz_1,\cdot) \rangle=f(\vz_1)$,
(iii) $\Phi_2(\cdot,\vz_2)\in\Banach^{p}$, and (iv) $\langle \Phi_2(\cdot,\vz_2),g \rangle=g(\vz_2)$,
for all $\vz_1,\vz_2\in\Omega$, all $f\in\Banach^{p}$, and all $g\in\Banach^{q}$, where $q:=p/(p-1)$, the dual space of $\Banach^{p}$ is isometrically equivalent to $\Banach^{q}$, and $\langle \cdot,\cdot \rangle$ represents the dual bilinear product of $\Banach^{p}\times\Banach^{q}$, that is,
$\langle f,g \rangle=\valpha^T\vbeta$ for $f=\valpha^T\vphi$ and $g=\vbeta^T\vphi$.
Moreover, \cite[Proposition 4.4]{XuYe2019} guarantees that $\Banach^{p}\subseteq\Cont(\Omega)$.
The G\^{a}teaux derivative of $\norm{\cdot}_{\Banach^{p}}$ at $f\neq0$ has the form
\begin{equation}\label{eq:GD-p}
d_G\norm{\cdot}_{\Banach^{p}}(f)=\sum_{k=1}^{\infty}\left(\frac{\alpha_k\abs{\alpha_k}^{p-2}}{\norm{\valpha}_p^{p-1}}\right)\phi_k,
\quad\text{for }f=\valpha^T\vphi\in\Banach^{p}.
\end{equation}
In Section \ref{sec:IntOptErr}, we will use the positive definite multi-kernel $\Phi_m$ to construct the interpolant and analyze its properties in the reproducing kernel Banach space $\Banach^{m/(m-1)}$ such as optimal recoveries and error analysis when $m$ is a positive even integer.

%---------------------------------------------------------------------------------------------------------------------
\section{Interpolations, Optimal Recoveries, and Error Analysis}\label{sec:IntOptErr}
%---------------------------------------------------------------------------------------------------------------------

In this section, we discuss how to construct the interpolant $s_m$ from the scattered data by the strictly positive definite multi-kernel $\Phi_m$
in Equation \eqref{eq:PDTK-sp}. Suppose that the data $(\vx_1,y_1),(\vx_2,y_2),\ldots,(\vx_n,y_n)$ compose of the pairwise distinct points $\left\{\vx_i\right\}_{i=1}^n\subseteq\Omega\subseteq\Rd$ and the values $\left\{y_i\right\}_{i=1}^n\subseteq\RR$
evaluated by some function $f\in\Cont(\Omega)$, that is,
\begin{equation}\label{eq:int-cond}
y_1:=f(\vx_1),y_2:=f(\vx_2),\ldots,y_n:=f(\vx_n).
\end{equation}
Let $\vy:=\left(y_1,y_2,\ldots,y_n\right)^T\in\RR^n$.
Let $\vA_m$ be a tensor defined in Equation \eqref{eq:Am-K} by the multi-kernel $K_m:=\Phi_m$ and the data points $\left\{\vx_i\right\}_{i=1}^n$.
Different from the classical kernel-based interpolations, the basis of $s_m$ composes of
$\Phi_m(\cdot,\vx_{i_1},\vx_{i_2},\cdots,\vx_{i_{m-1}})$ for $i_1,i_2,\ldots,i_{m-1}=1,2,\ldots,n$.
Let a tensor function
\[
\vB_m(\vx):=\left(\Phi_m(\vx,\vx_{i_1},\vx_{i_2},\ldots,\vx_{i_{m-1}})\right)_{i_1,i_2,\ldots,i_{m-1}=1}^{n,n,\ldots,n}\in T_{m-1,n},
\quad\text{for }\vx\in\Omega.
\]
Thus, $\vA_m=\left(\vB_m(\vx_i)\right)_{i=1}^n$.

%%%%%%%%%%%%%%%%%%%%%%%%%%%%%%%%%%%%%%%%%%%%%%%%%%%%%%%%%%%%%%%%%%%%%%%%%%%%%%%%%%%%%
\begin{theorem}\label{thm:PDTK-Int}
If $m$ is a positive even integer, then
$s_m$ has the from
\[
s_m(\vx):=\vB_m(\vx)\vc^{m-1},\quad\text{for }\vx\in\Omega,
\]
such that
\[
s_m(\vx_1)=y_1,s_m(\vx_2)=y_2,\ldots,s_m(\vx_n)=y_n,
\]
where the coefficients $\vc\in\RR^n$ are uniquely solved by the multi-linear system
\begin{equation}\label{eq:intr-2}
\vA_m\vc^{m-1}=\vy.
\end{equation}
\end{theorem}
%%%%%%%%%%%%%%%%%%%%%%%%%%%%%%%%%%%%%%%%%%%%%%%%%%%%%%%%%%%%%%%%%%%%%%%%%%%%%%%%%%%%%
\begin{proof}
Proposition \ref{pro:Phi-sPDT} guarantees that $\vA_m$ is a symmetric strictly positive definite tensor.
Thus, by Proposition \ref{pro:SPDT-MultiSys}, multi-linear system \eqref{eq:intr-2} has the unique solution $\vc$.
This assures that $s_m$ satisfies the interpolation conditions.
\end{proof}

\begin{remark}
If $s_m$ is constructed by the general positive definite multi-kernel $K_m$, then $\vA_m$ is a symmetric positive definite tensor such that
multi-linear system \eqref{eq:intr-2} still has the solution $\vc$. But, the solution $\vc$ may not be unique by Proposition \ref{pro:PDT-MultiSys} such that the interpolant $s_m$ may not be unique.
\end{remark}

Equation \eqref{eq:cond-phi} shows that $\Phi_m(\cdot,\vx_{i_1},\vx_{i_2},\cdots,\vx_{i_{m-1}})\in\Banach^{m/(m-1)}$;
hence $s_m\in\Banach^{m/(m-1)}$. Now we show the optimal recovery of $s_m$
%in $\Banach^{m/(m-1)}$
by the theorems in
\cite{XuYe2019}.

%%%%%%%%%%%%%%%%%%%%%%%%%%%%%%%%%%%%%%%%%%%%%%%%%%%%%%%%%%%%%%%%%%%%%%%%%%%%%%%%%%%%%
\begin{theorem}\label{thm:PDTK-Int-Opt}
If $m$ is a positive even integer, then $s_m$ is the minimizer of the norm-minimal interpolation
\begin{equation}\label{eq:opt-RKBS}
\min_{f\in\Banach^{m/(m-1)}}\norm{f}_{\Banach^{m/(m-1)}}\text{ subjected to }
f(\vx_i)=y_i\text{ for all }i=1,2,\ldots,n.
\end{equation}
\end{theorem}
%%%%%%%%%%%%%%%%%%%%%%%%%%%%%%%%%%%%%%%%%%%%%%%%%%%%%%%%%%%%%%%%%%%%%%%%%%%%%%%%%%%%%
\begin{proof}
If $\vy=0$, then $s_m=0$ and $s_m$ is the unique minimizer of minimization \eqref{eq:opt-RKBS}.
If $\vy\neq0$, then the minimizer of minimization \eqref{eq:opt-RKBS} is nonzero.
For convenience, we suppose that $\vy\neq0$.
\cite[Lemma~2.22]{XuYe2019} guarantees that minimization \eqref{eq:opt-RKBS} has the unique minimizer $s$ such that its Gateaux derivative $d_G\norm{\cdot}_{\Banach^{m/(m-1)}}(s)=\sum_{i=1}^n\beta_i\Phi_2(\vx_i,\cdot)$,
where the constants $\beta_1,\ldots,\beta_n\in\RR$. By the expansions of $\Phi_2$, the Gateaux derivative can be rewritten as
\begin{equation}\label{eq:PDTK-Int-Opt-1}
d_G\norm{\cdot}_{\Banach^{m/(m-1)}}(s)=\sum_{k=1}^{\infty}\left(\sum_{i=1}^n\beta_i\phi_k(\vx_i)\right)\phi_k.
\end{equation}
Combining Equations \eqref{eq:GD-p} and \eqref{eq:PDTK-Int-Opt-1}, the coefficients $\valpha$ of $s$ has the form
\begin{equation}\label{eq:PDTK-Int-Opt-2}
\frac{\alpha_k\abs{\alpha_k}^{m/(m-1)-2}}{\norm{\valpha}_{m/(m-1)}^{m/(m-1)-1}}=\sum_{i=1}^n\beta_i\phi_k(\vx_i),
\quad\text{for }k\in\NN.
\end{equation}
Let $c_i:=\norm{\valpha}_{m/(m-1)}^{m/(m-1)-1}\beta_i$ for all $i=1,2,\ldots,n$. Thus, Equation \eqref{eq:PDTK-Int-Opt-2} shows that
\begin{equation}\label{eq:PDTK-Int-Opt-3}
s(\vx)=\sum_{k=1}^{\infty}\left(\sum_{i=1}^nc_i\phi_k(\vx_i)\abs{\sum_{j=1}^nc_j\phi_k(\vx_j)}^{m-2}\right)\phi_k(\vx),\quad\text{for }\vx\in\Omega.
\end{equation}
Expanding of Equation \eqref{eq:PDTK-Int-Opt-3}, we have
\[
\begin{split}
s(\vx)&=\sum_{k=1}^{\infty}\sum_{i_1,\ldots,i_{m-1}=1}^{n,\ldots,n}\phi_k(\vx)\prod_{j=1}^{m-1}c_{i_j}\phi_n(\vx_{i_j})
=\sum_{i_1,\ldots,i_{m-1}=1}^{n,\ldots,n}\prod_{j=1}^{m-1}c_{i_j}\sum_{k=1}^{\infty}\phi_k(\vx)\prod_{l=1}^{m-1}\phi_n(\vx_{i_l})\\
&=\sum_{i_1,\ldots,i_{m-1}=1}^{n,\ldots,n}c_{i_1}c_{i_2}\cdots c_{i_{m-1}}\Phi_m(\vx,\vx_{i_1},\vx_{i_2},\cdots,\vx_{i_{m-1}})
=\vB_m(\vx)\vc^{m-1}.
\end{split}
\]
Since $s(\vx_i)=y_i$ for all $i=1,2,\ldots,n$, we have $\vA_m\vc^{m-1}=\vy$.
Moreover, Proposition \ref{thm:PDTK-Int} guarantees that the solution of multi-linear system \eqref{eq:intr-2} is unique. This assures that $s(\vx)=s_m(\vx)$.
\end{proof}

%%%%%%%%%%%%%%%%%%%%%%%%%%%%%%%%%%%%%%%%%%%%%%%%%%%%%%%%%%%%%%%%%%%%%%%%%%%%%%%%%%%%%
\begin{theorem}\label{thm:PDTK-Int-Norm}
If $m$ is a positive even integer, then
\[
\norm{s_m}_{\Banach^{m/(m-1)}}=\left(\vA_m\vc^m\right)^{1-1/m}.
\]
\end{theorem}
%%%%%%%%%%%%%%%%%%%%%%%%%%%%%%%%%%%%%%%%%%%%%%%%%%%%%%%%%%%%%%%%%%%%%%%%%%%%%%%%%%%%%
\begin{proof}
If $s_m=0$, then the proof is straightforward.
For convenience, we suppose that $s_m\neq0$. As in the proof of Theorem \ref{thm:PDTK-Int-Opt},
we have
\[
\norm{s_m}_{\Banach^{m/(m-1)}}=\langle s_m,d_G\norm{\cdot}_{\Banach^{m/(m-1)}}(s_m) \rangle
=\sum_{i=1}^n\beta_i\langle s_m,\Phi_2(\vx_i,\cdot) \rangle,
\]
and
\[
\beta_i=
\norm{\valpha}_{m/(m-1)}^{1-m/(m-1)}c_i
=\norm{s_m}_{\Banach^{m/(m-1)}}^{-1/(m-1)}c_i,\quad
\text{for }i=1,\ldots,n.
\]
By the reproducing properties of $\Banach^{m/(m-1)}$, we compute the norm
\[
\norm{s_m}_{\Banach^{m/(m-1)}}=\sum_{i=1}^n\beta_is_m(\vx_i)
=\norm{s_m}_{\Banach^{m/(m-1)}}^{-1/(m-1)}\sum_{i=1}^nc_i\vB_m(\vx_i)\vc^{m-1}.
\]
Therefore,
\[
\norm{s_m}_{\Banach^{m/(m-1)}}^{m/(m-1)}=\vA_m\vc^m.
\]
\end{proof}

For an approximate problem of $f(\vx_1)\approx y_1,\ldots,f(\vx_n)\approx y_n$, the estimate function $s$ is solved by the regularization
\begin{equation}\label{eq:reg-approx}
\min_{f\in\Banach^{m/(m-1)}}\sum_{i=1}^n\left(f(\vx_i)-y_i\right)^2+\sigma\norm{f}_{\Banach^{m/(m-1)}}^{m/(m-1)},
\end{equation}
where $\sigma$ is a positive parameter.
In the same manner of Theorems \ref{thm:PDTK-Int-Opt} and \ref{thm:PDTK-Int-Norm}, the minimizer of optimization \eqref{eq:reg-approx} has the form as $s=B_m\vc^{m-1}$, where
the coefficients $\vc$ are uniquely solved by the minimization
\[
\min_{\vc\in\RR^n}\norm{\vA_m\vc^{m-1}-\vy}_2^2+\sigma\vA_m\vc^m.
\]

Finally, we study the error analysis of $\abs{s_m(\vx)-f(\vx)}$ when $f\in\Banach^{m/(m-1)}$.
We define the generalized power function
\begin{equation}\label{eq:powerfun}
P_m(\vx):=\min_{\vartheta_1,\vartheta_2,\ldots,\vartheta_n\in\RR}
\norm{\Phi_2(\vx,\cdot)-\sum_{i=1}^n\vartheta_i\Phi_2(\vx_i,\cdot)}_{\Banach^{m}},
\quad\text{for }\vx\in\Omega.
\end{equation}
Let $\vx_0:=\vx$ and $\tilde{\vA}_m:=\left(\Phi_m(\vx_{i_1},\vx_{i_2},\cdots,\vx_{i_m})\right)_{i_1,i_2,\ldots,i_m=0}^{n,n,\ldots,n}\in T_{m,n+1}$.
As in the proof of Theorems \ref{thm:PDTK-Int-Opt} and \ref{thm:PDTK-Int-Norm}, we can compute
\[
P_m(\vx)=\min_{\vartheta_1,\vartheta_2,\ldots,\vartheta_n\in\RR}
\left(\tilde{\vA}_m\tilde{\vvartheta}^m\right)^{1/m},
\]
where $\tilde{\vvartheta}:=\left(1,-\vartheta_1,-\vartheta_2,\ldots,-\vartheta_n\right)^T\in\RR^{n+1}$.
Since $\Banach^{2}$ is a reproducing kernel Hilbert space with the reproducing kernel $\Phi_2$,
$P_2$ is equal to the classical power function in \cite[Section 11.1]{Wendland2005} such that
$P_2(\vx)=\left(\Phi_2(\vx,\vx)-\vA_2^{-1}\vB_2(\vx)^2\right)^{1/2}$.

%%%%%%%%%%%%%%%%%%%%%%%%%%%%%%%%%%%%%%%%%%%%%%%%%%%%%%%%%%%%%%%%%%%%%%%%%%%%%%%%%%%%%
\begin{theorem}\label{thm:PDTK-Int-Error}
If $m$ is a positive even integer and $f\in\Banach^{m/(m-1)}$ satisfies interpolation condition \eqref{eq:int-cond}, then
\[
\abs{s_m(\vx)-f(\vx)}\leq2\norm{f}_{\Banach^{m/(m-1)}}P_m(\vx),\quad\text{for }\vx\in\Omega.
\]
\end{theorem}
%%%%%%%%%%%%%%%%%%%%%%%%%%%%%%%%%%%%%%%%%%%%%%%%%%%%%%%%%%%%%%%%%%%%%%%%%%%%%%%%%%%%%
\begin{proof}
By the reproducing properties of $\Banach^{m/(m-1)}$, the interpolation conditions of $f$ and $s_m$ show that
\[
\langle f-s_m,\Phi_2(\vx_i,\cdot) \rangle=f(\vx_i)-s_m(\vx_i)=y_i-y_i=0,
\]
for $i=1,2,\ldots,n$.
Thus, we have
\begin{equation}\label{eq:PDTK-Int-Error-1}
\begin{split}
&f(\vx)-s_m(\vx)=\langle f-s_m,\Phi_2(\vx,\cdot) \rangle\\
=&\langle f-s_m,\Phi_2(\vx,\cdot) \rangle
-\sum_{i=1}^n\vartheta_i\langle f-s_m,\Phi_2(\vx_i,\cdot) \rangle\\
=&\langle f-s_m,\Phi_2(\vx,\cdot)-\sum_{i=1}^n\vartheta_i\Phi_2(\vx_i,\cdot) \rangle,
\end{split}
\end{equation}
for any $\vartheta_1,\vartheta_2,\ldots,\vartheta_n\in\RR$.
Since $\frac{m-1}{m}+\frac{1}{m}=1$,
Equation \eqref{eq:PDTK-Int-Error-1} shows that
\begin{equation}\label{eq:PDTK-Int-Error-2}
\abs{f(\vx)-s_m(\vx)}
\leq\norm{f-s_m}_{\Banach^{m/(m-1)}}
\norm{\Phi_2(\vx,\cdot)-\sum_{i=1}^n\vartheta_i\Phi_2(\vx_i,\cdot)}_{\Banach^{m}}.
\end{equation}
Moreover, since Theorem \ref{thm:PDTK-Int-Opt} guarantees that $\norm{f}_{\Banach^{m/(m-1)}}\geq\norm{s_m}_{\Banach^{m/(m-1)}}$,
we have
\begin{equation}\label{eq:PDTK-Int-Error-3}
\norm{f-s_m}_{\Banach^{m/(m-1)}}
\leq\norm{f}_{\Banach^{m/(m-1)}}+\norm{s_m}_{\Banach^{m/(m-1)}}
\leq2\norm{f}_{\Banach^{m/(m-1)}}.
\end{equation}
We conclude from Equations \eqref{eq:PDTK-Int-Error-2} and \eqref{eq:PDTK-Int-Error-3} that
\[
\abs{f(\vx)-s_m(\vx)}
\leq2\norm{f}_{\Banach^{m/(m-1)}}
\min_{\vartheta_1,\vartheta_2,\ldots,\vartheta_n\in\RR}\norm{\Phi_2(\vx,\cdot)-\sum_{i=1}^n\vartheta_i\Phi_2(\vx_i,\cdot)}_{\Banach^{m}},
\]
hence that the proof is completed by Equation \eqref{eq:powerfun}.
\end{proof}

Let the fill distance
$h:=\sup_{\vx\in\Omega}\min_{i=1,2,\ldots,n}\norm{\vx-\vx_i}_2$.
If $\Phi_2\in\Cont^{2\gamma}(\Omega\times\Omega)$ for $\gamma\in\NN$,
then \cite[Theorem 11.13]{Wendland2005} guarantees that
\begin{equation}\label{eq:PDTK-Int-Error-h-0}
P_2(\vx)\leq C_{\Omega}h^{\gamma},\quad\text{for }\vx\in\Omega,
\end{equation}
where $C_{\Omega}$ is a positive constant independent of $h$ and $\vx$.

%%%%%%%%%%%%%%%%%%%%%%%%%%%%%%%%%%%%%%%%%%%%%%%%%%%%%%%%%%%%%%%%%%%%%%%%%%%%%%%%%%%%%
\begin{theorem}\label{thm:PDTK-Int-Error-h}
If $m$ is a positive even integer, $f\in\Banach^{m/(m-1)}$ satisfies interpolation condition \eqref{eq:int-cond},
and $\Phi_2\in\Cont^{2\gamma}(\Omega\times\Omega)$ for $\gamma\in\NN$,
then
\[
\abs{f(\vx)-s_m(\vx)}\leq Ch^{\gamma},\quad\text{for }\vx\in\Omega,
\]
where $C$ is a positive constant independent of $h$ and $\vx$.
\end{theorem}
%%%%%%%%%%%%%%%%%%%%%%%%%%%%%%%%%%%%%%%%%%%%%%%%%%%%%%%%%%%%%%%%%%%%%%%%%%%%%%%%%%%%%
\begin{proof}
The main idea of the proof is to compute the upper bound of $P_m(\vx)$ by the fill distance $h$.
The constructions of $\Banach^{m}$ and $\Banach^{2}$ shows that
$\norm{f}_{\Banach^{m}}\leq\norm{f}_{\Banach^{2}}$
for $f\in\Banach^{m}\subseteq\Banach^{2}$; hence
\begin{equation}\label{eq:PDTK-Int-Error-h-1}
\norm{\Phi_2(\vx,\cdot)-\sum_{i=1}^n\vartheta_i\Phi_2(\vx_i,\cdot)}_{\Banach^{m}}
\leq
\norm{\Phi_2(\vx,\cdot)-\sum_{i=1}^n\vartheta_i\Phi_2(\vx_i,\cdot)}_{\Banach^{2}},
\end{equation}
for any $\vartheta_1,\vartheta_2,\ldots,\vartheta_n\in\RR$.
Substituting Equations \eqref{eq:PDTK-Int-Error-h-1} into \eqref{eq:powerfun}, we have
\begin{equation}\label{eq:PDTK-Int-Error-h-2}
P_m(\vx)\leq P_2(\vx).
\end{equation}
Combining Equations \eqref{eq:PDTK-Int-Error-h-0} and \eqref{eq:PDTK-Int-Error-h-2}, we also have
\[
P_m(\vx)\leq C_{\Omega}h^{\gamma}.
\]
Let $C:=2\norm{f}_{\Banach^{m/(m-1)}}C_{\Omega}$. Therefore, the proof is completed by Theorem \ref{thm:PDTK-Int-Error}.
\end{proof}

Theorem \ref{thm:PDTK-Int-Error-h} guarantees that $s_m(\vx)\to f(\vx)$ when $h\to0$ for $\vx\in\Omega$, and moreover, $s_m\to f$ uniformly when $h\to0$.
A slight change in the proof of Equation \eqref{eq:PDTK-Int-Error-h-2} shows that $P_{m_1}(\vx)\leq P_{m_2}(\vx)$
when $m_1\geq m_2$.
We will study how the error decreases when $m$ increases in our future work.

\section*{Acknowledgments}

The author would like to acknowledge support for this project from the Natural Science Foundation of
China (12071157 and 12026602) and
the Natural Science Foundation of Guangdong Province (2019A1515011995 and 2020B1515310013).




%\bibliographystyle{amsalpha}
\bibliographystyle{plain}
\bibliography{RKBSRef}


%\begin{thebibliography}{00}
%%% Text of bibliographic item
%
%\bibitem{}
%
%\end{thebibliography}

\end{document}